\documentclass[11pt,reqno]{amsart}

 \usepackage{amsmath,amsthm,amsfonts}
\usepackage{latexsym,mathabx,txfonts}
\usepackage{amssymb}

\newcommand{\R}{\mathbb{R}}

\renewcommand{\bar}{\overline}

\def\calL{{\mathcal L}}

\numberwithin{equation}{section}

\newtheorem{thm}{Theorem}[section]

\newtheorem{lem}[thm]{Lemma}
\newtheorem{prop}[thm]{Proposition}

\theoremstyle{remark}
\newtheorem{rem}{Remark}[section]

\newcommand{\Del}[1]{}

\def\eps{\varepsilon}

\begin{document}

\title[Unstable type II blow up]{Instability of type II blow up for the quintic nonlinear wave equation on $\R^{3+1}$.}

\author{Joachim Krieger, Joules Nahas}



\thanks{Support of the Swiss National Fund for
the first author is gratefully acknowledged. He would like to thank the University of Chicago for its hospitality in August 2012}

%The abstract of your paper
\begin{abstract}
We prove that the blow up solutions of type II character constructed by Krieger-Schlag-Tataru \cite{KST} as well as Krieger-Schlag \cite{KS1} are unstable in the energy topology in that there exist open data sets whose closure contains the data of the preceding type II solutions and such that data in these sets lead to solutions scattering to zero at time $t = +\infty$. 
\end{abstract}

\maketitle 

\section{Introduction}

We consider the quintic focussing wave equation on $\R^{3+1}$, of the form 
\begin{equation}\label{eq:foccrit}
\Box u = u^5,\,\Box  = \partial_t^2 -\triangle
\end{equation}
in the radial context, i. e. $u(t, x) = v(t, |x|)$. 
This equation is of energy critical and focussing type and serves as a convenient model for more complicated energy critical models, such as Wave Maps in $2+1$ dimensions with positively curved targets, or Yang-Mills equations in $4+1$ dimensions as well as related problems of Schr\"odinger type. In fact, for example recent progress on \eqref{eq:foccrit} in \cite{DoKr} has led to analogous progress for the energy critical focussing NLS in $3+1$ dimensions, \cite{OP}. 
The focussing character of \eqref{eq:foccrit} leads to finite time blow up, which is most easily manifested by the explicit solutions of ODE type 
\[
u(t, x) = \frac{(\frac{3}{4})^{\frac{1}{4}}}{(T-t)^{\frac{1}{2}}}
\]
for arbitrary $T$. Truncating the data of these solutions at time $t = 0$ to force finiteness of $\int_{\R^3}[u_t^2(0, \cdot) + |\nabla_x u(0, \cdot)|^2]\,dx$, one easily verifies that 
\[
\lim_{t\rightarrow T}\int_{\R^3}[u_t^2(t, \cdot) + |\nabla_x u(t, \cdot)|^2]\,dx = +\infty
\]
One says the blow up is of type I. By contrast, a finite time blow up solution with 
\[
\limsup_{t\rightarrow T}\int_{\R^3}[u_t^2(t, \cdot) + |\nabla_x u(t, \cdot)|^2]\,dx < +\infty
\]
where $T$ is the blow up time is called of type II. Existence of the latter type of solution for \eqref{eq:foccrit}
 is rather subtle and appears to have first been accomplished in \cite{KST}, see also \cite{KS1}, and Hillairet-Rapha\"el's paper \cite{HR} for more stable blow up solutions in the $4+1$-dimensional context. The works \cite{KST}, \cite{KS1} show that denoting 
\[
u_{\lambda}(t, x): = \lambda^{\frac{1}{2}} u(\lambda t, \lambda x),\,\lambda>0
\]
problem \eqref{eq:foccrit} admits type II blow up solutions (of energy class) of the form 
\begin{equation}\label{eq:typeIIblow}
u(t, x) = W_{\lambda(t)}(x) + \eps(t, x),\,\lambda(t) = (-t)^{-1-\nu},\,t\in [-t_0, 0),\,t_0\lesssim 1, 
\end{equation}
for $\nu>0$, with $W(x)$ denoting the ground state static solution 
\[
W(x) = \frac{1}{\big(1+\frac{|x|^2}{3}\big)^{\frac{1}{2}}}
\]
 More precisely, the solutions constructed in \cite{KST}, \cite{KS1} admit a precise description of the radiation term $\eps(t, x)$ inside the light cone $\{r\leq |t|\}$ of the form 
\[
\eps(t, x) = O\left ( \lambda^{\frac{1}{2}}(t)\frac{R}{(\lambda(t) |t|)^2} \right ),\,R = \lambda(t)|x|,
\]
and furthermore $\eps(t,\cdot)\in \dot{H}^1$ with 
\[
\|(\eps, \eps_{t})\|_{(\dot{H}^1\times L^2)(|x|\leq t)}\lesssim (\lambda(t) |t|)^{-\frac{1}{2}}
\]
By contrast, outside of the light cone, we can only assert that 
\[
\|\nabla_{t,x}\eps\|_{L^2(|x|\geq t)}\leq \delta_*
\]
where we may arrange for $\delta_*$ to be arbitrarily small. Indeed, this is consistent with the fact proved in \cite{DKM1} that type II blow up solutions must have energy strictly larger than that of the ground state. 
\\
We also mention that analogous {\it{infinite time blow up solutions}} were constructed in \cite{DoKr}. See also \cite{HR} for type II blow up with a different rate for the energy critical NLW in $4+1$ dimensions. 
\\

The remarkable series of papers \cite{DKM1} - \cite{DKM4} recently gave a complete classification of the possible type II solutions, on finite or infinite time intervals, in the radial context for \eqref{eq:foccrit}. These works show that any type II solution decouples as a sum of dynamically rescaled ground states $\pm W$ at diverging scales, plus an error that remains regular at blow up time (or radiates to zero in the infinite time case). In these works, it is intimated that all such type II solutions ought to be unstable in the energy topology, and in fact ought to constitute the boundary of both the set of solutions existing globally and scattering to zero, as well as those blowing up of type I. Indeed, it is only the latter two which are readily observable in numerical experiments.  
\\
The recent work \cite{CNLW3}  gives a rather precise description of the instability of the static solution $W$ with respect to a suitably strong topology. 
\\
Here, we show that the solutions constructed in \cite{KST}, \cite{KS1} are unstable in the energy topology, provided $\eps$ has sufficiently small energy.
Specifically, we have 
\begin{thm}\label{thm:main}There exists $\delta_*>0$ with the following property: let $u(t,x)$ be one of the type II blow up solutions constructed in \cite{KST}, \cite{KS1}
\begin{equation}\label{eq:generaltypeII}
u(t, x) = W_{\lambda(t)}(x) + \eps(t, x),\,\lambda(t) = (-t)^{-1-\nu}
\end{equation}
satisfying the a priori condition 
\[
\limsup_{t\in [-t_0, 0)}\|\nabla_{t,x}\eps(t, \cdot)\|_{L_x^2} < \delta_*,\,t_0\lesssim 1
\]
Then there exists an open set of data $U$ in the energy topology at time $-t_0$, with 
\[
\big(u(-t_0, \cdot),  u_t(-t_0, \cdot)\big) \in \bar{U},
\]
and such that all data in $U$ lead to solutions existing globally and scattering to zero in forward time. Also, there is an open set of data $V$ with 
\[
\big(u(-t_0, \cdot),  u_t(-t_0, \cdot)\big) \in \bar{V},
\]
with the property that all data in $V$ lead to finite time blow up. 

\end{thm}

\begin{rem}It remains to show that any open subset of the set of blow up data $V$ contains data leading to type I blow up.
\end{rem}

The preceding theorem is quite similar in character to an earlier seminal result by Merle, Rapha\"el and Szeftel\cite{MeRaSz}, where a similar instability phenomenon for the pseudo-conformal blow up solutions constructed by Bourgain-Wang is shown in the context of the $L^2$-critical NLS equation. This is in fact in some sense a more delicate phenomenon than the one exhibited in our paper, since there is no {\it{exponential instability}} that can be exploited in the context of \cite{MeRaSz}. 
\\
We also note that while the blow up solutions whose instability is shown in our paper correspond to a continuum of rates, such a phenomenon has not been observed for the $L^2$-critical NLS, and in fact the Bourgain-Wang solutions all have the same blow up rate. This is in contrast to the {\it{energy-critical focussing NLS}}, where a continuum of blow up rates is possible at least at time $t=+\infty$(and presumably also at finite time), see \cite{OP}. 
Nonetheless, for all these models, including the one in our present paper, there is conjectured to be a {\it{stable blow up rate}} which is numerically observable, as well as a  number of unstable ones. 
For our equation $\Box u = -u^5$, this is conjectured to be the rate exhibited by the explicit ODE-type blow up. In fact, stability of the explicit ODE-blow up solutions has been shown in \cite{DoSch}. 
\\
One may wonder if the {\it{continuum of rates}} possible for blow up solutions for energy critical problems is tied to the topology one works in, and more precisely, if imposing $C^\infty$-smoothness on the data will remove most of these, leaving a {\it{quantized set}} of blow up solutions corresponding to special quantized values quantized of the blow up speed. The lowest blow up rate would then conjecturally correspond to stable blow up, while all the higher ones would be unstable, according to the mechanism exhibited in our paper. Such a picture was developed for the harmonic map heat flow in \cite{RaSchw}. 
\\

The idea of the proof of Theorem~\ref{thm:main} is as follows: Fixing a smallness parameter $\delta_1>0$, which quantifies the smallness of the perturbation applied to the explicit type II blow up solution,  we intend to apply the ejection type argument of \cite{CNLW1} at some time $-t_1>-t_0$, $t_1 = t_1(\delta_1)$, by applying a suitable small excitation in the unstable direction of the linearization around $W$. The reason that we can only implement the ejection at time $-t_1$ comes from the fact that we need to ensure that the scaling parameter $\lambda(t)$ experiences sufficiently small marginal changes\footnote{This may be the reason why the type II blow up here is also referred to as slow blow up} in re-scaled coordinates past time $-t_1$, depending on $\delta_1$.  Specifically, we shall arrange that the solution exits a suitable small neighborhood of $\mathcal{S}: = \{W_{\lambda}\}_{\lambda>0}$ in forward time past $t = -t_1$ very quickly. On the other hand, we also need to arrange that this perturbed solution {\it{remains stable on $[-t_0, -t_1]$}} and indeed is essentially $\delta_1$-close to $u$ at $t = -t_0$. In fact, the solutions thus constructed will be a one parameter family (parametrized by $\delta_1$), but a simple perturbative argument then gives the desired open set of solutions $U$ {\it{at time $t = -t_0$}}. 
In fact, it is assuring that the perturbed solution remains close to $u(t, \cdot)$ on $[-t_0, -t_1]$ which causes most of the difficulties, and forces us to exploit the precise structure of the solutions constructed in \cite{KST}, \cite{KS1}. 
\\
We observe here that the construction in this paper appears to be of much wider applicability, and in particular ought to be able to handle instability of blow up solutions with rates much closer to $t^{-1}$, such as the logarithmic type corrections considered in \cite{HR}. 

\section{Constructing a stable solution on $[-t_0, -t_1]$.}

Our point of departure are the solutions 
\[
u_{II}(t, x): = W_{\lambda(t)}(x) + \eps(t, x),\,t\in [-t_0, 0)
\]
constructed in \cite{KST}, \cite{KS1}. We aim at perturbing these on an  interval $[-t_0, -t_1]$ with $t_1 = \delta_1\ll t_0$. Let 
\[
\mathcal{H}: = -\triangle - 5W^4
\]
and let $g_0$ be its unique negative eigenstate, see e. g. \cite{KS}. Introducing the re-scaled variables 
\[
\tau = \frac{1}{\nu}(-t)^{-\nu},\, R = \lambda(t)r,
\]
it is then natural to consider the perturbed approximate solution 
\[
\tilde{u}_{II}(t, x): = W_{\lambda(t)}(x) + \eps(t, x) + \eta_{hyp}(t, x)
\]
where we put 
\[
\eta_{hyp}(t, x) = a e^{-k_d(\tau_1 - \tau)}g_0(R),\,\tau_i =  \frac{1}{\nu}t_i^{-\nu},\,i = 0,1,
\]
with $a = \pm \delta_1^{\nu N}$ for some large $N$ to be chosen. Also, $-k_d^2$ is the unique negative eigenvalue of $\mathcal{H}$, with corresponding eigenmode $g_0$. Our first problem is to show that this can be completed to an exact solution on $[-t_0, -t_1]$, by adding a suitable $\eta(t, x)$. In effect, we shall work with 
\[
\tilde{\eta}(\tau, R):= R\eta(t, x)
\]
To find this function, we employ the Fourier theoretic methods developed in \cite{KST}. Thus, using terminology developed there, we write 
\[
\tilde{\eta}(\tau, R) = x_d(\tau)\phi_d(R) + \int_0^\infty x(\tau, \xi)\phi(R, \xi)\rho(\xi)\,d\xi,
\] 
where the function $\phi_d(R) = Rg_0(R)$ is the unique negative eigenmode associated with the operator 
\[
\mathcal{H}: = -\partial_R^2 - 5W^4(R),
\] 
while $\phi(R, \xi)$ constitutes the distorted Fourier basis. Also, $\rho(\xi)$ denotes the spectral measure associated with this operator.  By the corresponding Fourier inversion theorem, we have 
\[
x_d(\tau) = \int_0^\infty \tilde{\eta}(\tau, R) \phi_d(R)\,dR,\,x(\tau, \xi) = \int_0^\infty \tilde{\eta}(\tau, R) \phi(R, \xi)\,dR
\]
We shall use the $H^2_{dR}$ norm to control $\tilde{\eta}(\tau, R)$, which shall be handy in the section on the ejection process. From \cite{KST} we recall that 
\[
\|\tilde{\eta}(\tau, \cdot)\|_{H^2_{dR}}\lesssim \|\langle\xi\rangle x(\tau, \xi)\|_{L^2_{d\rho}} + |x_d(\tau)|
\]
\[
\|\tilde{\eta}(\tau, \cdot)\|_{H^1_{dR}}\lesssim \|\langle\xi\rangle^{\frac{1}{2}} x(\tau, \xi)\|_{L^2_{d\rho}} + |x_d(\tau)|
\]
In this section, we shall write 
\[
\|x(\tau, \cdot)\|_{S}: = \|\langle\xi\rangle x(\tau, \xi)\|_{L^2_{d\rho}}
\]
Also, for the source terms, we use the norm 
\[
\|x(\tau, \cdot)\|_{N}: =  \|\langle\xi\rangle^{\frac{1}{2}} x(\tau, \xi)\|_{L^2_{d\rho}}
\]

\begin{prop}\label{prop:stable}There exists $\tilde{\eta}(\tau, \cdot)$, $\tau\in [\tau_0, \tau_1]$, with 
\[
\|\tilde{\eta}(\tau, \cdot)\|_{H^2_{dR}}\lesssim \delta_1^{\nu}\tau^{-(N-1)},\,|P_{g_0}\tilde{\eta}(\tau, \cdot)|\lesssim \delta_1^{\nu}\tau^{-N}Rg_0(R),
\]
such that 
\[
\tilde{u}_{II} + \eta
\]
solves \eqref{eq:foccrit} on $[-t_0, -t_1]$. 
\end{prop}
\begin{proof} 
We first consider to what extent the expression $\tilde{u}_{II}(t, x)$ is an approximate solution of \eqref{eq:foccrit}. Observe that 
\begin{align*}
(-\partial_t^2 +\triangle)\tilde{u}_{II} + \tilde{u}_{II}^5 = (-\partial_t^2 +\triangle)\eta_{hyp} - 5W_{\lambda(t)}^4\eta_{hyp} + A =: B+A,
\end{align*}
where we put 
\[
A: = \sum_{\substack{j+k+l=5,\\j<4,\,l\neq 0}}C_{j,k}W_{\lambda(t)}^j\eps^k \eta_{hyp}^l
\]
We commence by bounding the various constituents: 
\\

{\it{The terms $B$:}} This can be written as a linear combination of terms of the form 
\begin{equation}\label{eq:B1}\begin{split}
&\big(t\lambda(t)\big)'' e^{-k_0(\tau_1 - \tau)} ag_0(R),\, \big(t\lambda(t)\big)' \frac{\lambda'(t)}{\lambda(t)}e^{-k_0(\tau_1 - \tau)}a(R\partial_R) g_0(R),\\&\big(\frac{\lambda'(t)}{\lambda(t)}\big)^2e^{-k_0(\tau_1 - \tau)}a(R\partial_R)^2 g_0
\end{split}\end{equation}
Labeling these terms $B_i$, $i = 1,2,3$, we then infer the following key bounds
\[
\|R\lambda^{-2}(\tau)B_i\|_{H^1_{dR}}\lesssim \delta_1^{\nu N}\tau^{-1} e^{-k_0(\tau_1 - \tau)}\lesssim \delta_1^{\nu}\tau^{-N} e^{-k_0(\tau_1 - \tau)},\,i=1,2,3
\]
{\it{The terms $A$:}} These are bounded by using 
\[
\eta_{hyp} = O(\delta_1^{\nu N}e^{-k_0(\tau_1 - \tau)}) = O(\delta_1^{\nu}\tau^{-(N-1)}e^{-k_0(\tau_1 - \tau)}).
\]
Thus we get 
\[
\|R\lambda^{-2}(\tau)W_{\lambda(t)}^3\eta_{hyp}^2\|_{H^1_{dR}}\lesssim \lambda^{-\frac{1}{2}}(\tau)\delta_1^{2\nu}e^{-2k_0(\tau_1 - \tau)}\tau^{-2(N-1)}
\]
and further 
\[
\|R\lambda^{-2}(\tau)\eta_{hyp}^5\|_{H^1_{dR}}\lesssim \delta_1^{5\nu}\lambda^{-2}(\tau)e^{-5k_0(\tau_1 - \tau)}\tau^{-5(N-1)}
\]
Also, we get 
\begin{align*}
&\|R\lambda^{-2}(\tau)\eps^{4}\eta_{hyp}\|_{H^1_{dR}}\lesssim \delta_1^{\nu}\tau^{-N-7}e^{-k_0(\tau_1 - \tau)},\\&\|R\lambda^{-2}(\tau)W_{\lambda(t)}^3\eps\eta_{hyp}\|_{H^1_{dR}}\lesssim \delta_1^{\nu}\tau^{-N-1}e^{-k_0(\tau_1 - \tau)}
\end{align*}
Observe that we have $\eps(\tau, R) = O(\lambda^{\frac{1}{2}}\frac{R}{(\lambda |t|)^2})$ on $\tau\lesssim R$ but on the region $|\tau|\gtrsim R$ we get $\|\eta_{hyp}\|_{H^1(R\gtrsim \tau)}\lesssim e^{-c\tau}$. 
Let us denote $A+B: = e_0$. Then we obtain the following equation for $\eta$: 
\begin{equation}\label{eq:etawave}
\mathcal{D}^2\tilde{\eta} + \beta_{\nu}(\tau)\mathcal{D}\tilde{\eta} + \calL\tilde{\eta} = \kappa^{-2}(\tau)\big[5(\tilde{u}_{II}^4 - u_0^4)\tilde{\eta} + RN(\tilde{u}_{II}, \tilde{\eta}) + R e_{0}\big]
\end{equation}
where we have introduced the notation 
\[
\calL = -\partial_R^2 - 5W^4(R),\,\mathcal{D} = \partial_{\tau} + \beta_{\nu}(\tau)(R\partial_R - 1),
\]
\[
N(\tilde{u}_{II}, \tilde{\eta}) = (\tilde{u}_{II} + \frac{\tilde{\eta}}{R})^5 - R\tilde{u}_{II}^5 - 5\tilde{u}_{II}^4\frac{\tilde{\eta}}{R}
\]

Our task is to solve \eqref{eq:etawave} on $[\tau_0, \tau_1]$ subject to the boundary conditions
\[
\tilde{\eta}(\tau_1, \cdot) = \partial_{\tau}\tilde{\eta}(\tau_1, \cdot) = 0
\]
Fortunately, the operator on the left of \eqref{eq:etawave} admits a convenient parametrix, so we can solve this equation by recourse to the Fourier representation. 
Passing to the Fourier side, we obtain the equation 
\begin{equation}\label{eq:transport}
\big(\mathcal{D}_{\tau}^2 + \beta_{\nu}(\tau)\mathcal{D}_{\tau} + \underline{\xi}\big)\underline{x}(\tau, \xi) = R(\tau, \underline{x}) + f(\tau, \underline{\xi}),\,\underline{\xi} =  \left(\begin{array}{c}\xi_d\\ \xi\end{array}\right)
\end{equation}
where we write 
\[
\tilde{\eta}(\tau, R) = x_d(\tau)\phi_d(R) + \int_0^\infty x(\tau, \xi)\phi(R, \xi)\rho(\xi)\,d\xi,\, \underline{x} = \left(\begin{array}{c}x_d(\tau)\\ x(\tau, \xi)\end{array}\right)
\]
and we have 
\[
 \mathcal{D}_{\tau} = \partial_{\tau} + \beta_{\nu}(\tau)\mathcal{A},\,\mathcal{A} = \left(\begin{array}{cc}0&0\\0&\mathcal{A}_c\end{array}\right)
 \]
as well as 
\begin{equation}\label{eq:Rterms}
R(\tau, \underline{x})(\xi) = \big(-4\beta_{\nu}(\tau)\mathcal{K}\mathcal{D}_{\tau}\underline{x} - \beta_{\nu}^2(\tau)(\mathcal{K}^2 + [\mathcal{A}, \mathcal{K}] + \mathcal{K} + \frac{\beta_{\nu}'}{\beta_{\nu}^2}\mathcal{K})\underline{x}\big)(\xi)
\end{equation}
\[
\mathcal{A}_c = -2\xi\partial_{\xi} - (\frac{5}{2}+\frac{\rho'(\xi)\xi}{\rho(\xi)}),\,\mathcal{K} = \left(\begin{array}{cc}\mathcal{K}_{dd}&\mathcal{K}_{dc}\\\mathcal{K}_{cd}&\mathcal{K}_{cc}\end{array}\right),\,\beta_{\nu}(\tau) = \frac{\dot{\lambda}(\tau)}{\lambda(\tau)},
\]
where the symbols $\mathcal{K}_{dd}$ etc  are operators defined in \cite{KST}. Finally, $f(\tau, \underline{\xi}) $ represents the Fourier transform of the source terms. 
 \begin{equation}\label{eq:fterms}
 f(\tau, \underline{\xi}) = \mathcal{F}\big( \lambda^{-2}(\tau)\big[5(\tilde{u}_{II}^4 - u_0^4)\tilde{\eta} + RN(\tilde{u}_{II}, \tilde{\eta}) + R e_{0}\big]\big)\big(\underline{\xi}\big)
 \end{equation}
 The rapid decay of the variable $\tilde{\eta}(\tau, R)$ allows us to solve \eqref{eq:transport} via a direct iteration scheme, essentially as in \cite{KST}. Specifically, we use 
 \begin{equation}\label{eq:transportiter}
\big(\mathcal{D}_{\tau}^2 + \beta_{\nu}(\tau)\mathcal{D}_{\tau} + \underline{\xi}\big)\underline{x}_j(\tau, \xi) = R(\tau, \underline{x}_{j-1}) + f_{j-1}(\tau, \underline{\xi}),\,\underline{\xi} =  \left(\begin{array}{c}\xi_d\\ \xi\end{array}\right),\,j\geq 1,
\end{equation}
with 
 \begin{equation}
 f_{j-1}(\tau, \underline{\xi}) = \mathcal{F}\big(\lambda^{-2}(\tau)\big[5(\tilde{u}_{II}^4 - u_0^4)\tilde{\eta} + RN(\tilde{u}_{II}, \tilde{\eta}) + R e_{0}\big]\big)\big(\underline{\xi}\big),\,j\geq 2
 \end{equation}
as well as $f_0: = 0$. \\
To proceed, we observe that the linear inhomogeneous problem 
\[
\big(\mathcal{D}_{\tau}^2 + \beta_{\nu}(\tau)\mathcal{D}_{\tau} + \underline{\xi}\big)\underline{x}(\tau, \xi) = f(\tau, \xi)
\]
can be solved completely explicitly (imposing vanishing data at infinity).
 In fact, the way we have set things up, we can use (See \cite{KS1})
 \begin{equation}\label{eq:para}\begin{split}
x(\tau, \xi) &= \xi^{-\frac{1}{2}}\int_{\tau}^{\tau_1} \frac{\lambda^{\frac{3}{2}}(\tau)}{\lambda^{\frac{3}{2}}(\sigma)}\frac{\rho^{\frac{1}{2}}(\frac{\lambda^{2}(\tau)}{\lambda^{2}(\sigma)}\xi)}{\rho^{\frac{1}{2}}(\xi)}\sin\big[\lambda(\tau)\xi^{\frac{1}{2}}\int_{\tau}^{\sigma}\lambda^{-1}(u)\,du\big]f(\sigma, \frac{\lambda^{2}(\tau)}{\lambda^{2}(\sigma)}\xi)\,d\sigma\\
&=: \int_{\tau}^{\tau_1} U(\tau, \sigma, \xi)]f(\sigma, \frac{\lambda^{2}(\tau)}{\lambda^{2}(\sigma)}\xi)\,d\sigma
\end{split}\end{equation}
for the continuous spectral part, while we have the implicit equation
\begin{equation}\label{eq:x_d}
x_{d}(\tau) = \int_{\tau}^{\tau_1} H_d(\tau, \sigma)\big(f_d(\sigma) - \beta_{\nu}(\sigma)\partial_{\sigma}x_d(\sigma)\big)\,d\sigma,\,H_d(\tau, \sigma) = -\frac{1}{2}|\xi_d|^{-\frac{1}{2}}e^{-|\xi_d|^{\frac{1}{2}}|\tau-\sigma|}
\end{equation}
Then we have 

\begin{lem}\label{lem:converge} For $N$ large enough and $t_0$ small enough, we have the a priori bounds 
\begin{align*}
&\|x_j(\tau, \cdot)\|_{S}\lesssim \delta_1^{\nu}\tau^{-(N-1)},\,\|\mathcal{D}_{\tau}x_j(\tau, \cdot)\|_{N}\lesssim \delta_1^{\nu}\tau^{-N},\, |x_{j,d}(\tau)|\lesssim \delta_1^{\nu}\tau^{-N},\\&|\partial_{\tau}(x_{j,d})(\tau)|\lesssim \delta_1^{\nu}\tau^{-N}
\end{align*}
Moreover, for the differences, we have 
\begin{align*}
&\|(x_j - x_{j-1})(\tau, \cdot)\|_{S}\lesssim (\frac{1}{N})^{j-1}\delta_1^{\nu}\tau^{-(N-1)},\,|(x_{j,d} - x_{j-1,d})(\tau)|\lesssim (\frac{1}{N})^{j-1}\delta_1^{\nu}\tau^{-N},\\
&\|\mathcal{D}_{\tau}(x_j - x_{j-1})(\tau, \cdot)\|_{S}\lesssim (\frac{1}{N})^{j-1}\delta_1^{\nu}\tau^{-N}
\end{align*}

\end{lem}
\begin{proof} We commence by observing as in \cite{KST} that 
\[
\sup_{\tau>\tau_0>0}\tau^{N-1}\big\|\int_{\tau}^{\tau_1} U(\tau, \sigma, \xi)]f(\sigma, \frac{\lambda^{2}(\tau)}{\lambda^{2}(\sigma)}\xi)\,d\sigma\big\|_{S}\lesssim \frac{1}{N}\sup_{\tau>\tau_0}\tau^{N+1}\|f(\tau, \cdot)\|_{N}
\]
Here we lose two powers of decay in $\tau$ due to the singular kernel and the fact that we integrate over $\tau$. 
For the first iterate, we can exploit the exponential decay to lose only one power of $\tau$.  \\
For the time derivative, we get 
\[
\sup_{\tau>\tau_0>0}\tau^{N}\big\|\mathcal{D}_{\tau}\int_{\tau}^{\tau_1} U(\tau, \sigma, \xi)]f(\sigma, \frac{\lambda^{2}(\tau)}{\lambda^{2}(\sigma)}\xi)\,d\sigma\big\|_{S}\lesssim \frac{1}{N}\sup_{\tau>\tau_0}\tau^{N+1}\|f(\tau, \cdot)\|_{N}
\]
where we lose one power of $\tau$ due to the integration. Thus an additional exponential weight in $f$ improves decay by one power. \\
We also use the bound
\[
\big\|R(\tau, \underline{x})(\xi)\big\|_{N}\lesssim \tau^{-2}\big(\|x(\tau, \cdot)\|_{S} + |x_d(\tau)|\big) + \tau^{-1}\big(\|\mathcal{D}_{\tau}x(\tau, \cdot)\|_{S} + |\partial_{\tau}x_d(\tau)|\big)
\]
 which follows easily from the estimates in \cite{KST}; here the implicit constant is independent of $N$.  To conclude the proof of the lemma, we now need to bound the contributions from the source terms in $f_{j-1}$. For simplicity, we suppress the subscript in the sequel. 
\\

{\it{The contribution from $5\lambda^{-2}(\tau)(\tilde{u}_{II}^4 - u_0^4)\tilde{\eta}$.}} Recalling the definition of $\tilde{u}_{II}$, we have to bound the following list of terms (omitting intermediate terms): 
\begin{equation}\label{eq:list1}
\lambda^{-2}(\tau)\mathcal{F}\big(W_{\lambda(t)}^3\eta_{hyp}\tilde{\eta}\big),\,\lambda^{-2}(\tau)\mathcal{F}\big(W_{\lambda(t)}^3\eps\tilde{\eta}\big),\,\lambda^{-2}(\tau)\mathcal{F}\big(\eta_{hyp}^4\tilde{\eta}\big),\,\lambda^{-2}(\tau)\mathcal{F}\big(\eps^4\tilde{\eta}\big)
\end{equation}
For the first term, we get 
\begin{align*}
\big\|\lambda^{-2}(\tau)\mathcal{F}\big(W_{\lambda(t)}^3\eta_{hyp}\tilde{\eta}\big)\big\|_{N}\lesssim \lambda^{-2}(\tau)\big\|W_{\lambda(t)}^3\eta_{hyp}\tilde{\eta}\big\|_{H^1_{dR}}&\lesssim \lambda^{-\frac{1}{2}}(\tau)\|\eta_{hyp}\|_{H^2_{R^2 dR}}\|\tilde{\eta}\|_{H^1_{dR}}\\
&\lesssim \delta_1^{2\nu}\lambda^{-\frac{1}{2}}(\tau)\tau^{-2(N-1)}
\end{align*}
which is more than enough since $2(N-1)\geq N+1$ for $N\geq 3$. \\
For the second term in \eqref{eq:list1} we use 
\[
\lambda^{-\frac{1}{2}}\eps(\tau, R) = \chi_{R\leq \nu\tau}\lambda^{-\frac{1}{2}}\eps(\tau, R) + \chi_{R> \nu\tau}\lambda^{-\frac{1}{2}}\eps(\tau, R)
\]
Then, use that 
\[
\|\chi_{R\leq \nu\tau}\lambda^{-\frac{1}{2}}\langle R\rangle^{-\frac{3}{2}-}\eps(\tau, R)\|_{H^1_{dR}}\lesssim \tau^{-2},
\]
which gives 
\begin{align*}
&\big\|\lambda^{-2}(\tau)W_{\lambda(t)}^3\chi_{R\leq \nu\tau}\eps\tilde{\eta}\big\|_{H^1_{dR}}\\&\lesssim \|\lambda^{-\frac{3}{2}}\langle R\rangle^{\frac{3}{2}+}W_{\lambda(t)}^3\|_{\dot{H}^1_{dR}\cap L^\infty_{dR}}\|\chi_{R\leq \nu\tau}\lambda^{-\frac{1}{2}}\langle R\rangle^{-\frac{3}{2}-}\eps(\tau, R)\|_{H^1_{dR}}\|\tilde{\eta}\|_{H^1_{dR}}\\
&\lesssim \delta_1^{\nu}\tau^{-N-1}
\end{align*}
Next, consider the contribution of $\chi_{R> \nu\tau}\lambda^{-\frac{1}{2}}\eps(\tau, R)$. Here we simply use the explicit decay of $W(R)$ to get 
\begin{align*}
\big\|\lambda^{-2}(\tau)W_{\lambda(t)}^3\chi_{R>\nu\tau}\eps\tilde{\eta}\big\|_{H^1_{dR}}&\lesssim \tau^{-3}\|\lambda^{-\frac{1}{2}}\eps(\tau, \cdot)\|_{\dot{H}^{1}_{R^2 dR}\cap L^6_{R^2 dR}}\|\tilde{\eta}\|_{H^1_{dR}}\\
&\lesssim \tau^{-N-2}\delta_1^{\nu},
\end{align*}
which is again better than what we need. \\
The third term in \eqref{eq:list1} is better than the first (due to the exponential decay of $\eta_{hyp}$) and hence omitted. The last term in \eqref{eq:list1} is a bit more delicate: in the interior of the light cone, the explicit expansion of $\eps(\tau, R)$ implies that 
\[
\big\|\chi_{R\leq \nu\tau}\lambda^{-\frac{1}{2}}\eps(\tau, R)\big\|_{H^1_{dR}}\lesssim \tau^{-\frac{1}{2}}, 
\]
and so we find 
\begin{align*}
\big\|\lambda^{-2}(\tau)\eps^4\tilde{\eta}\big\|_{H^1_{dR}}\lesssim \|\lambda^{-\frac{1}{2}}\eps(\tau, R)\|_{H^1_{dR}}^4\|\tilde{\eta}\|_{H^1_{dR}}\lesssim \delta_1^{\nu}\tau^{-N-1}
\end{align*}
On the outside of the light cone, we need to estimate 
\[
\big\|\chi_{R>\nu\tau}\lambda^{-2}(\tau)(\nabla_{R}\eps) \eps^3\tilde{\eta}\big\|_{L^2_{dR}} + \big\|\chi_{R>\nu\tau}\lambda^{-2}(\tau)\eps^4\nabla_R\tilde{\eta}\big\|_{L^2_{dR}} +  \big\|\chi_{R>\nu\tau}\lambda^{-2}(\tau) \eps^4\tilde{\eta}\big\|_{L^2_{dR}}
\]
Here we use the estimate 
\[
\big\|\chi_{R>\nu\tau}\lambda^{-\frac{1}{2}}(\tau)\eps(\tau, R)\big\|_{L^\infty_{dR}}\lesssim \tau^{-\frac{1}{2}}\|\lambda^{-\frac{1}{2}}(\tau)\eps(\tau, R)\|_{\dot{H}^{1}_{R^2 dR}\cap L^6_{R^2 dR}}
\]
as well as 
\[
\|\chi_{R>\nu\tau}\frac{1}{R}\tilde{\eta}(\tau, R)\|_{L^\infty}\lesssim \tau^{-\frac{1}{2}}\|\tilde{\eta}(\tau, R)\|_{H^1_{dR}}
\]
Thus we obtain 
\begin{align*}
&\big\|\chi_{R>\nu\tau}\lambda^{-2}(\tau)(\nabla_{R}\eps) \eps^3\tilde{\eta}\big\|_{L^2_{dR}}\\&\lesssim \|\lambda^{-\frac{1}{2}}(\tau)\nabla_{R}\eps\|_{R^2 dR}\big\|\chi_{R>\nu\tau}\lambda^{-\frac{1}{2}}(\tau)\eps(\tau, R)\big\|_{L^\infty_{dR}}^3\|\chi_{R>\nu\tau}\frac{1}{R}\tilde{\eta}(\tau, R)\|_{L^\infty}\\
&\lesssim \delta_1^{\nu}\tau^{-N-1}
\end{align*}
and further 
\begin{align*}
&\big\|\chi_{R>\nu\tau}\lambda^{-2}(\tau)\eps^4\nabla_R\tilde{\eta}\big\|_{L^2_{dR}} +  \big\|\chi_{R>\nu\tau}\lambda^{-2}(\tau) \eps^4\tilde{\eta}\big\|_{L^2_{dR}}\\&\lesssim \big\|\chi_{R>\nu\tau}\lambda^{-\frac{1}{2}}(\tau)\eps(\tau, R)\big\|_{L^\infty_{dR}}^4\|\tilde{\eta}(\tau, R)\|_{H^1_{dR}}\\
&\lesssim \delta_1^{\nu}\tau^{-N-1}
\end{align*}

{\it{The contribution of the terms $RN(\tilde{u}_{II}, \tilde{\eta})$.}} Since we have 
\[
N(\tilde{u}_{II}, \tilde{\eta}) = (\tilde{u}_{II} + \frac{\tilde{\eta}}{R})^5 - R\tilde{u}_{II}^5 - 5\tilde{u}_{II}^4\frac{\tilde{\eta}}{R}
\]
we have to estimate terms of the form 
\[
R\lambda^{-2}(\tau)\tilde{u}_{II}^{5-j}(\frac{\tilde{\eta}}{R})^j,\,j\geq 2,
\]
Since we have from Sobolev's embedding
\[
\big\|\frac{\tilde{\eta}}{R}\big\|_{L^\infty}\lesssim \big\|\frac{\tilde{\eta}}{R}\big\|_{H^2_{R^2 dR}}\sim \|\tilde{\eta}\|_{H^2_{dR}},
\]
and also $\lambda^{-\frac{1}{2}}\tilde{u}_{II}\in \dot{H}^1_{dR}\cap L^\infty$, we find 
\begin{align*}
\big\|R\lambda^{-2}(\tau)\tilde{u}_{II}^{5-j}(\frac{\tilde{\eta}}{R})^j\big\|_{H^1_{dR}}\lesssim \delta_1^{2\nu}\tau^{-2(N-1)},\,j\geq 2
\end{align*}
which is again more than enough. This concludes the proof of the lemma, up to the statement about the differences and the better decay for the discrete spectral part. The gains of $N^{-1}$ follow from integrating the weights $\tau^{-(N-1)}$ (and better). The better decay for the discrete spectral part is a consequence of the exponential decay of the kernel $H_d(\tau, \sigma)$ in \eqref{eq:x_d}. 
\end{proof}

The proof of the proposition follows by a simple iteration argument using the lemma. 

\end{proof}

\section{Ejection past time $t = -t_1$.} We next need to show that the solution constructed above with 
\[
u(t, x) = W_{\lambda(t)}(x) + \eps(t, x) + \eta_{hyp}(t, x) + \eta(t, x),\,t\in [-t_0, -t_1]
\]
leads to a controlled exit past time $t = -t_1$ from a suitable neighborhood of $\mathcal{S}: = \{W_{\lambda}\}_{\lambda>0}$. Specifically, we shall re-scale by $\lambda(t_1)$ and shift the new time origin to time $-t_1$, which changes the solution to 
\begin{equation}\label{eq:tildeu}
\tilde{u}(t, x) = W_{\frac{\lambda(-t_1+t\lambda^{-1}(-t_1))}{\lambda(-t_1)}}(x) + \eps_{\lambda^{-1}(t_1)}(-t_1 + t\lambda^{-1}(-t_1), x) + \tilde{\eta},
\end{equation}
and now we need to track the evolution of $\tilde{\eta}$ in {\it{forward time}} but on a scale of size at most $|\log \delta_1|$. In fact, the equation for $\tilde{\eta}$ becomes 
\[
(-\partial_t^2 -\mathcal{H})\tilde{\eta} = 5(W^4 - W_{\frac{\lambda(-t_1+t\lambda^{-1}(-t_1))}{\lambda(-t_1)}}^4)\tilde{\eta} + N(\tilde{\eta}),\,\mathcal{H} = -\triangle -5W^4
\]
with data at the new time origin corresponding to $t = t_1$ given by 
\[
\big(\tilde{\eta}, \partial_t\tilde{\eta}\big)|_{t=0} = \big(\lambda^{-\frac{1}{2}}(t_1)\eta_{hyp}(\tau_1, \cdot) ,\,\lambda^{-\frac{1}{2}}(t_1)\partial_{\tau}\eta_{hyp}(\tau_1, \cdot)\big)
\]
where we write 
\begin{align*}
-N(\tilde{\eta})& = \big(W_{\frac{\lambda(-t_1+t\lambda^{-1}(-t_1))}{\lambda(-t_1)}}(x) + \eps_{\lambda^{-1}(t_1)}(t, x) + \tilde{\eta}\big)^5\\& - \big(W_{\frac{\lambda(-t_1+t\lambda^{-1}(-t_1))}{\lambda(-t_1)}}(x) + \eps_{\lambda^{-1}(t_1)}(t, x)\big)^5 - 
5W_{\frac{\lambda(-t_1+t\lambda^{-1}(-t_1))}{\lambda(-t_1)}}^4\tilde{\eta}
\end{align*}
Re-labelling 
\[
u_{II}: = W_{\frac{\lambda(-t_1+t\lambda^{-1}(-t_1))}{\lambda(-t_1)}}(x) + \eps_{\lambda^{-1}(t_1)}(t, x), 
\]
we then obtain the equation 
\begin{equation}\label{eq:etaforward}
(-\partial_t^2 -\mathcal{H})\tilde{\eta} = 5(W^4 - W_{\frac{\lambda(-t_1+t\lambda^{-1}(-t_1))}{\lambda(-t_1)}}^4)\tilde{\eta} + A + B
\end{equation}
where we put 
\[
A: = 5\big[(W_{\frac{\lambda(-t_1+t\lambda^{-1}(-t_1))}{\lambda(-t_1)}}(x) + \eps_{\lambda^{-1}(t_1)}(t, x))^4 - W_{\frac{\lambda(-t_1+t\lambda^{-1}(-t_1))}{\lambda(-t_1)}}^4(x)\big]\tilde{\eta}
\]
\[
B: = \sum_{j=2}^5 C_j u_{II}^{5-j}\tilde{\eta}^j
\]
Note that while in the preceding section the coordinate change $t\rightarrow\tau \sim \lambda(t) t,\,|x| = r\rightarrow R = \lambda(t)r$ was {\it{time dependent}}, here we have a {\it{time-independent}} coordinate change 
\[
t\rightarrow \frac{\lambda(-t_1+t\lambda^{-1}(-t_1))}{\lambda(-t_1)},\,r\rightarrow |x| =\lambda(-t_1)r
\]

We then split 
\[
\tilde{\eta} = \delta(t)g_0(|x|) + \tilde{\eta}_c,\,\tilde{\eta}_c = P_{g_0^{\perp}}\tilde{\eta}
\]
Observe that 
\[
 \delta(0) = \lambda^{-\frac{1}{2}}(-t_1)a,\,\tilde{\eta}_c(0, \cdot) = 0
\]
Then we have the following 
\begin{lem}\label{lem:eject} There is some $\delta_0>0$ sufficiently small {\it{but independent of $\delta_1$}} with the following property: denoting 
\[
b: = a\lambda(-t_1)^{-\frac{1}{2}},
\]
if $b e^{k_d T}\leq \delta_0$, $0<T$,
\[
\delta(t)\sim b e^{k_d t},\,\|\tilde{\eta}_{c}(t, \cdot)\|_{H^2_x}\ll |b|e^{k_d t}
\]
 \end{lem}
  \begin{proof} We use a simple bootstrap argument, exploiting the fact that this is a perturbative statement. Thus we make a bootstrap assumption of the form 
 \begin{equation}\label{eq:bootstr}
 |\delta(t)|\leq 2|b| e^{k_d t},\,\|\tilde{\eta}_{c}(t, \cdot)\|_{H^2_x}\leq \frac{2|b|}{K}e^{k_d t}
 \end{equation}
 for some suitable large $K$(absolute constant, which is large but small enough compared to $\delta_0^{-1}$), and then improve these bounds by a factor $2$. We start with the bounds for $\tilde{\eta}_c$, for which we have the equation 
 \begin{equation}\label{eq:etaforwardeta_c}
(-\partial_t^2 -\mathcal{H})\tilde{\eta}_c = P_{g_0^{\perp}}\big(5(W^4 - W_{\frac{\lambda(-t_1+t\lambda^{-1}(-t_1))}{\lambda(-t_1)}}^4)\tilde{\eta} + A + B\big) =: F
\end{equation}
From Duhamel's principle, we infer 
\begin{align*}
\tilde{\eta}_c(t, \cdot) =
-\int_0^t\frac{\sin([t-s]\sqrt{\mathcal{H}})}{\sqrt{\mathcal{H}}}P_{g_0^{\perp}}F(s, \cdot)\,ds 
\end{align*}
We estimate each of the constituents of $F(s, \cdot)$,  making use of the bootstrap assumptions: 
\\

{\it{(i) The contribution of $F_1: = P_{g_0^{\perp}}\big(5(W^4 - W_{\frac{\lambda(-t_1+t\lambda^{-1}(-t_1))}{\lambda(-t_1)}}^4)\tilde{\eta}\big)$.
}} Here we use the algebraic structure of the scaling parameter $\lambda(t)$ to infer 
\[
\big|\frac{\lambda(-t_1+t\lambda^{-1}(-t_1))}{\lambda(-t_1)} - 1\big| = O \left (\frac{t}{t_1\lambda(t_1)} \right )
\]
Then restricting $t$ to $[0, C|\log\delta_1|]$, we get the bound 
\begin{align*}
\big\|\int_0^t\frac{\sin([t-s]\sqrt{\mathcal{H}})}{\sqrt{\mathcal{H}}}F_1(s, \cdot)\,ds\big\|_{H_x^2}&\lesssim |b|\int_0^t (t-s)\frac{s}{t_1\lambda(t_1)}e^{k_d s}\,ds\\
&\ll\frac{|b|}{K}e^{k_d t}
\end{align*}

Next, we continue with the contributions of the terms $A$ and $B$: 
\\

{\it{(ii) The contribution of the term $A$}}. These terms fall under the general form 
\[
F_{2,j}: = W_{\frac{\lambda(-t_1+t\lambda^{-1}(-t_1))}{\lambda(-t_1)}}^{4-j} \eps_{\lambda^{-1}(t_1)}^j\tilde{\eta},\,0<j\leq 4
\]
Observe as in the preceding section that 
\[
\eps_{\lambda^{-1}(t_1)}(-t_1 + t\lambda^{-1}(t_1), x) = O \left (\frac{R}{(\lambda(-\tilde{t}) \tilde{t})^2} \right ),\,\tilde{t}: = -t_1 + t\lambda^{-1}(t_1)
\]
provided
\[
\lambda(\tilde{t}) |\tilde{t}| \sim t_1\lambda(t_1)\gg t,\,R\lesssim t_1\lambda(-t_1)
\]
Thus imposing these restrictions we get the uniform bounds 
\[
\|\chi_{R\lesssim t_1\lambda(-t_1)}F_{2,j}(s, \cdot)\|_{H_x^1}\lesssim  [t_1\lambda(t_1)]^{-2}|b|e^{k_d s}
\]
On the other hand, in the region $R\gtrsim t_1\lambda(-t_1)$ we can use the uniform $\dot{H}^1_{R^2 dR}\cap L^6_{R^2 dR}$-bound on $\eps_{\lambda^{-1}(t_1)}$, 
from which we infer in particular the bound 
\[
\big\|\chi_{R\gtrsim  t_1\lambda(-t_1)} \eps_{\lambda^{-1}(t_1)}\big\|_{L^\infty}\lesssim \big( t_1\lambda(-t_1)\big)^{-\frac{1}{2}}
\]
and then again 
\[
\|\chi_{R\gtrsim t_1\lambda(-t_1)}F_{2,j}(s, \cdot)\|_{H_x^1}\lesssim  [t_1\lambda(t_1)]^{-2}|b|e^{k_d s}
\]
In summary, we obtain 
\[
\big\|\int_0^t\frac{\sin([t-s]\sqrt{\mathcal{H}})}{\sqrt{\mathcal{H}}}P_{g_0^{\perp}}F_{2,j}(s, \cdot)\,ds\big\|_{H_x^2}\ll \frac{|b|}{K}e^{k_d t},\,j = 1,2,3,4
\]
provided $t\ll t_1\lambda(-t_1)$, which is the case due to $t = O(|\log\delta_1|)$ . 
\\

{\it{(ii) The contribution of the term $B$}}. Here we use the bound
\[
\|u_{II}^{5-j}\tilde{\eta}^j(s, \cdot)\|_{H_x^1}\lesssim |b|^2 e^{2k_d s}, 
\]
which leads to 
\[
\big\|\int_0^t\frac{\sin([t-s]\sqrt{\mathcal{H}})}{\sqrt{\mathcal{H}}}P_{g_0^{\perp}}\big[u_{II}^{5-j}\tilde{\eta}^j(s, \cdot)\big]\,ds\big\|_{H_x^2}\ll \frac{|b|}{K}e^{k_d t},\,j = 2,3,4,5
\]
since $|b| e^{k_d t}\leq \delta_0\ll 1$ by assumption. 
This concludes the bootstrap for the continuous spectral part $\tilde{\eta}_c$. \\
We next turn to the discrete part, i. e. the evolution of the function $\delta(t)$. As in \cite{CNLW1}, we can write 
\[
\delta(t) = (2k_d)^{-\frac{1}{2}}[n_+(t) + n_-(t)],
\]
where we have 
\[
n_{\pm}(t) = (\frac{k_d}{2})^{\frac{1}{2}}b e^{\pm k_d t} + \int_0^t e^{\pm k_d(t-s)}\langle F(s, \cdot), g_0\rangle\,ds
\]
It remains to bound the integral term in the right. We control the various ingredients of $F$: 
\\

{\it{(i)}} For the contribution of $f_1(t): = \langle 5(W^4 - W_{\frac{\lambda(-t_1+t\lambda^{-1}(-t_1))}{\lambda(-t_1)}}^4)\tilde{\eta}, g_0\rangle$, we have 
\[
\big|\int_0^t e^{\pm k_d(t-s)}f_1(s)\,ds\big|\lesssim |b|e^{k_d t}\int_0^t \frac{s}{[t_1\lambda(t_1)]^2}\,ds \ll \frac{|b|}{K}e^{k_d t}
\]
{\it{(ii)}} For the contribution of $f_2(t): = \langle A, g_0\rangle$, we obtain the exact same bound by exploiting 
\[
\big|\langle W_{\frac{\lambda(-t_1+t\lambda^{-1}(-t_1))}{\lambda(-t_1)}}^{4-j} \eps_{\lambda^{-1}(t_1)}^j\tilde{\eta}, g_0\rangle\big|\lesssim \frac{|b| e^{k_d t}}{[t_1\lambda(t_1)]^2},\,j = 1,2,3,4
\]
In effect, here one obtains exponential temporal decay in the region $R\gtrsim t_1\lambda(-t_1)$, due to the exponentially decaying $g_0$. 
\\
{\it{(iii)}} For the contribution of $f_3(t): = \langle B, g_0\rangle$, we have the bound 
\[
|f_3(t)|\lesssim b^2 e^{2k_d t},
\]
whence 
\[
\int_0^t e^{\pm k_d(t-s)} |f_3(s)|\,ds\lesssim |b|^2 e^{2k_d t}\ll \frac{|b|}{K}e^{k_d t}
\]
where we used the assumption $|b|e^{k_d t}\leq |b|e^{k_d T}\leq \delta_0$. 
\\

This concludes the bootstrap, and the lemma easily follows from this. 
\end{proof}

The precise statement which shall imply Theorem~\ref{thm:main} is now furnished by 

\begin{prop}\label{prop:key} Let $\tilde{u}(t, x)$ be the solution considered in \eqref{eq:tildeu}. Then there exists a time $T>0$ with 
\[
|b|e^{k_d T}\sim \delta_0\gg \delta_*
\]
and a decoupling 
\[
\tilde{u}(T, \cdot) = W_{\alpha_T} + \tilde{v}_{\alpha_T},\,|1-\alpha_T|\ll 1,
\]
and such that 
\[
\langle  \tilde{v}_{\alpha_T}, \Lambda^* g_{\alpha_T}\rangle = 0,\,\Lambda = R\partial_R + \frac{1}{2}
\]
and furthermore
\[
\langle \tilde{v}_{\alpha_T}, g_{\alpha_T}\rangle \sim be^{k_d T}
\]
Here $g_{\alpha_T}$ is the negative eigenmode of the linearization around $W_{\alpha_T}$. 
\end{prop}

\begin{proof} The argument here is essentially the same as in \cite{CNLW3}. In light of \eqref{eq:tildeu}, we have to satisfy the vanishing condition 
\begin{equation}\label{eq:van}
\langle W_{\frac{\lambda(-t_1+T\lambda^{-1}(-t_1))}{\lambda(-t_1)}} -  W_{\alpha_T} + \eps_{\lambda^{-1}(t_1)}(-t_1 + T\lambda^{-1}(-t_1), x) + \tilde{\eta}(T, \cdot),\, \Lambda^* g_{\alpha_T}\rangle = 0
\end{equation}
with $T$ chosen to satisfy $|b|e^{k_d T}\sim \delta_0$. Since we have the bound 
\[
\big\|\eps_{\lambda^{-1}(t_1)}(-t_1 + T\lambda^{-1}(-t_1), x) + \tilde{\eta}(T, \cdot)\big\|_{\dot{H}^1_{R^2 dR}}\lesssim \delta_0\ll 1,
\]
and the non-degeneracy condition 
\[
\big|\langle \partial_{\lambda}W_{\lambda}|_{\lambda = 1},\,\Lambda^*g_{\lambda}|_{\lambda=1}\rangle\big|\sim 1
\]
holds, see \cite{CNLW3}, an application of the implicit function theorem implies the existence of $\alpha_T$ $\delta_0$-close to $1$ such that \eqref{eq:van} is satisfied. 
Furthermore, since 
\[
\big|\langle W_{\frac{\lambda(-t_1+T\lambda^{-1}(-t_1))}{\lambda(-t_1)}} -  W_{\alpha_T}, g_{\alpha_T}\rangle\big|\sim \delta_0^2,
\]
we also find 
\begin{align*}
&\langle W_{\frac{\lambda(-t_1+T\lambda^{-1}(-t_1))}{\lambda(-t_1)}} -  W_{\alpha_T} + \eps_{\lambda^{-1}(t_1)}(-t_1 + T\lambda^{-1}(-t_1), x) + \tilde{\eta}(T, \cdot),\, g_{\alpha_T}\rangle\\
&\sim be^{k_d T},
\end{align*}
as desired.  
\end{proof}

\section{Proof of Theorem~\ref{thm:main}}

Having fixed a very small $\delta_1>0$ and constructed the solution $\tilde{u}$ on the time interval $[0, T]$ as in the preceding proposition, and recalling that this solution, when re-scaled by $\lambda(-t_1)$, can be extended from time $-t_1$ backwards to time $-t_0$ as described in the last section but one (i. e. the solution $u_{II}+\eta$ constructed there), a simple continuous dependence argument reveals that perturbing the data  of $u_{II}+\eta$ at time $t = -t_0$ by a sufficiently small amount in the energy topology, we obtain another solution which extends to time $t = -t_1$ and such that re-scaling and shifting the time origin at to time $t = t_1$ as in the preceding section, the corresponding solution also extends all the way up to time $t = T$, and satisfies the conclusion of Proposition~\ref{prop:key}. Now let $\tilde{u}(t, x)$ be as in the preceding section. Then denoting $\mathcal{S}: = \{(W_{\lambda}, 0)\}_{\lambda>0}$, we have\footnote{We denote $u[t] = \big(u(t, \cdot), \partial_t u(t, \cdot)\big)$}
\[
\text{dist}_{\dot{H}^1\times L^2}\big(\tilde{u}[0], \mathcal{S}\cup -\mathcal{S}\big)\lesssim \delta_*;
\]
In fact, this can be arranged by picking $\delta_1$ small enough. On the other hand, by Proposition~\ref{prop:key}, we have 
\[
\text{dist}_{\dot{H}^1\times L^2}\big(\tilde{u}[T], \mathcal{S}\cup -\mathcal{S}\big)\sim \delta_0\gg \delta_*
\]
This is a consequence of \cite{CNLW1}, Lemma 2.2. But then equation (3.44) as well as Proposition 5.1, Proposition 6.2 in \cite{CNLW1} imply that picking $a<0$ in the definition of $\eta_{hyp}$ in the last section but one leads to a solution $\tilde{u}(t, x)$ which exists globally in forward time and scatters towards zero. On the other hand, picking $a>0$ leads to $\tilde{u}(t, x)$ blowing up in finite forward time.   
Since we have 
\[
\big\| \tilde{u}_{II} + \eta(t, x) - \big(W_{\lambda(t)}+\eps(t, x)\big)[-t_0]\big\|_{\dot{H^1}\times L^2}\lesssim \delta_1^{\nu}
\]
according to Proposition~\ref{prop:stable} and $\delta_1>0$ was arbitrary, Theorem~\ref{thm:main} is proved.

% Enter the first author's name and address:
\centerline{\scshape Joachim Krieger }
\medskip
{\footnotesize
% please put the address of the first author
 \centerline{B\^{a}timent des Math\'ematiques, EPFL}
\centerline{Station 8, 
CH-1015 Lausanne, 
  Switzerland}
  \centerline{\email{joachim.krieger@epfl.ch}}
} % Do not forget to end the {\footnotesize by the sign }

\medskip

\centerline{\scshape Joules Nahas}
\medskip
{\footnotesize
% please put the address of the first author
 \centerline{B\^{a}timent des Math\'ematiques, EPFL}
\centerline{Station 8, 
CH-1015 Lausanne, 
  Switzerland}
  \centerline{\email{joules.nahas@epfl.ch}}
} % Do not forget to end the {\footnotesize by the sign }

\bigskip

\end{document}